\documentclass[14pt, twoside]{article}

\usepackage[T2A]{fontenc}
\usepackage[utf8]{inputenc}
\usepackage{graphicx}
\usepackage{placeins}
\usepackage{bmpsize}
\usepackage{amsfonts, amsmath, amssymb, amsthm}
\usepackage{floatflt}
\usepackage{graphics}
\usepackage[usenames]{color}
\usepackage{colortbl}
\usepackage{url}

\DeclareGraphicsExtensions{.pdf,.png,.jpg}

\theoremstyle{plain}
\newtheorem{theorem}{Theorem}
\newtheorem{lemma}{Lemma}
\newtheorem{corollary}{Corollary}

\newtheorem*{conj}{Conjecture}

\theoremstyle{remark}
\newtheorem*{remark}{Remark}

\theoremstyle{definition}
\newtheorem*{defin}{Definition}

{}
\renewcommand{\geq}{\geqslant}
\renewcommand{\leq}{\leqslant}

\def\_phi{\varphi}

\def\co{\mathrm{conv}}

\date{}
\RequirePackage{graphicx}
\title{Minimum number of partial triangulations}
\author{Andrey Kupavskii\footnote{G-SCOP, CNRS, University Grenoble-Alpes, France and Moscow Institute of Physics and Technology, Russia; Email: {\tt kupavskii@ya.ru}}, Aleksei Volostnov\footnote{Moscow Institute of Physics and Technology, Russia; Email: {\tt gyololo@yandex.ru}}, Yury Yarovikov\footnote{Moscow Institute of Physics and Technology, Russia; Email: {\tt yu-rovikov@yandex.ru}},
}
\begin{document}
\maketitle
\begin{abstract}
We show that the number of partial triangulations of a set of $n$ points on the plane is at least the $(n-2)$-nd Catalan number. This is tight for convex $n$-gons. We also describe all the equality cases.
\end{abstract}
Let $M$ be a set of points on the plane and let $M'$ be the set of all vertices of the convex hull ${\rm conv}(M)$ of $M$. A {\it full triangulation} of $M$ is a triangulation of ${\rm conv}(M)$ such that the set of vertices of the triangulation is $M$. A {\it partial triangulation} is a triangulation of ${\rm conv}(M)$ such that the set of its vertices $V$ satisfies $M'\subset V\subset M$.

Let $W_n$ be the number of (full or partial) triangulations of a convex $n$-gon, and put $W_2 = 1$ for convenience. It is easy to see that $W_n = c_{n-2}$, where $c_i$ is the $i$-th Catalan number (see Lemma~\ref{lemma-verify-equality-case} for a proof). Recall that $c_n = \frac 1{n+1}{2n\choose n},$  $c_{n+1} =\sum_{i=0}c_ic_{n-i}$, and that the first few numbers starting from $c_0$ are $\{1,1,2,5,14,42,\ldots\}$.

The following conjecture was raised by Emo Welzl during the Oberwolfach meeting on Discrete geometry in September 2020 (cf. Problem 10 in \cite{problist}):

\begin{conj}
Convex $n$-gons minimize the number of partial triangulations among any point sets in general position. In other words, any set of $n$ points on the plane in general position has at least $W_n$ partial triangulations.
\end{conj}

Interestingly, convex $n$-gons are not the only examples of point sets with minimum number of partial triangulations. Another $n$-point set with $W_{n-2}$ partial triangulations is the so-called {\it double circle}: a set consisting of a convex $n/2$-gon and $n/2$ points chosen in the interior of the convex hull, each corresponding to one of the sides of the outer $n/2$-gon and very close to the midpoint of the respective side. Double circles are the examples with the smallest known number of full triangulations: $12^{n/2}$ for $n$-point sets. We refer to a paper of E. Welzl \cite{W} for an account of related counting problems, and to a paper of U. Wagner and E. Welzl \cite{WW} for the study of flip graphs of full and partial triangulations. \\

In this short note, we show that the above conjecture is indeed true and determine all possible equality cases. For a set of points $M$ and a point $P\in M$ we say that $P$ is an {\it interior point} of $M$ if it does not lie on the convex hull of $M$. We say that in interior point $P$ is {\it close} to a side $QR$ of the convex hull of $M$, if any triangle $P'QR$ with $P'\in M$ contains $P.$ Note that, for any side $QR$ of the convex hull of $M$, there may be at most one point close to it.
\begin{defin}
We call a set of points $M$ \textit{quasi-convex} if each interior point of $M$ is close to some side of $\co (M)$.
\end{defin}

\begin{theorem}\label{thm1}
Any set of $n$ points on the plane in general position has at least $W_n$ partial triangulations. Moreover, a set of $n$ points has exactly $W_n$ triangulations if and only if it is quasi-convex.
\end{theorem}
That is, all extremal examples in some sense ``interpolate'' between a set in convex position and the double circle.

In the next section, we will give the proof of the conjecture without the equality cases. In Section~\ref{sec2} we will give the proof of the theorem with the equality cases. The two proofs are similar, but the second builds on the first one and is a bit more technical. This is the reason we give both proofs.

\section{Proof of the conjecture}\label{sec1}

\begin{lemma}\label{lemw}
For integer $x,y\geq2$ we have
$$W_xW_y\leq W_{x+y-2}.$$

\end{lemma}

\begin{proof}
Let us rewrite the inequality using Catalan numbers. Let $\alpha = x-2$, $\beta = y-2$. We have $W_x W_y = c_{\alpha}c_{\beta}$, $W_{x+y-2} = c_{\alpha + \beta}$, and the inequality states $c_\alpha c_\beta \le c_{\alpha+\beta}$. By one of the definitions of Catalan numbers, $c_{\alpha}$ and $c_{\beta}$ are the numbers of balanced bracket sequences of $\alpha$ and $\beta$ pairs of brackets, respectively. A concatenation of two such bracket sequences is, in turn, a balanced bracket sequence of $\alpha+\beta$ brackets. There are exactly $c_{\alpha + \beta}$ such bracket sequences and each two sequences obtained by such concatenation are clearly different (for fixed $\alpha$ and $\beta$). This concludes the proof.
\end{proof}

\begin{corollary}
For integer $k_1,k_2,\ldots,k_m\geq2$ we have
$$W_{k_1}\ldots W_{k_m}\leq W_{k_1+\ldots + k_m-2(m-1)}.$$
\end{corollary}

For this section we tacitly assume that all the angles that we work with are smaller than $\pi$. Given points $A,B,C$ on the plane, consider the angle $BAC$ and points $X,Y$ inside the angle, all $5$ points being in general position.
We say that a point $X\in M$ is \emph{to the left} of a point $Y\in M$ if $\angle BAX < \angle BAY$, and is \emph{to the right} if $\angle BAX > \angle BAY$ (cf. Fig.~\ref{fig2}).
The terminology comes from the situation when the bisector of $\angle BAC$ is a vertical ray going downwards, $B$ is on the left ray of the angle, and $C$ is on the right ray. Still, note that it is not the same as stating that the $x$-coordinate of $X$ is less than that of $Y$.
In what follows, we use the term {\it polyline} for a polygonal chain (broken line).

\begin{figure}
    \centering
\includegraphics[width=0.6\linewidth]{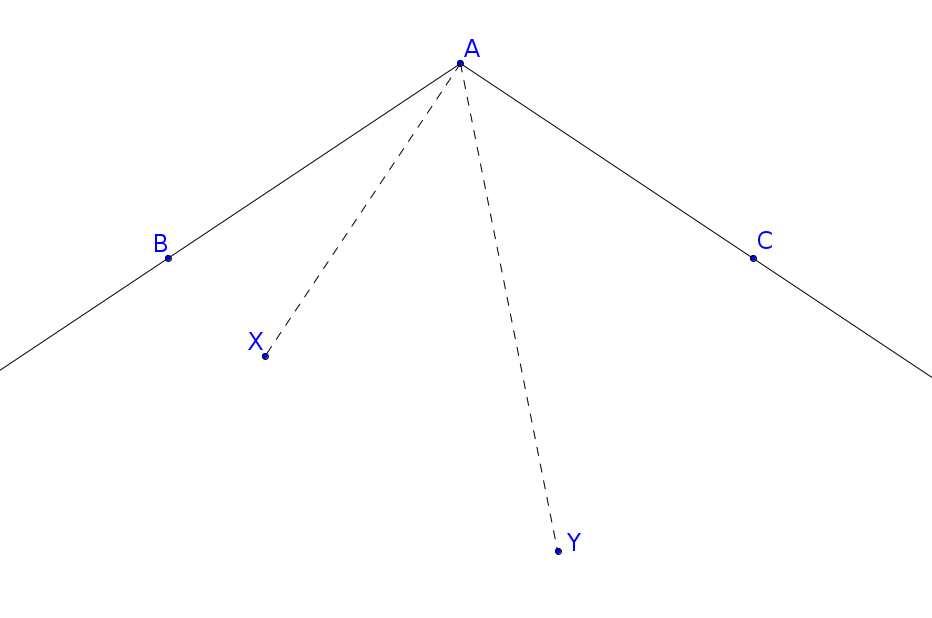}
    \caption{Point $X$ is to the left of the point $Y$.}
    \label{fig2}
\end{figure}

Fix an angle $ABC$ and set of points $M$ inside the angle, all points being in general position. The set $M = \{P_1, P_2, \ldots, P_n\}$ is enumerated left to right. Also, put $P_0 = B$, $P_{n+1} = C$ for convenience.
Consider an arbitrary polyline $L=P_{a_0}P_{a_1}\ldots P_{a_{m+1}}$, $a_0 = 0<a_1<\ldots<a_m<a_{m+1} = n+1$, going from left to right.

The next definition is  describing the position of points in $M$ with respect to $L$. Let us first give it informally. Imagine that there are nails in each of the points $M\cup \{B,C\}$ and there is a rubber band stretched from $B$ go $C$ that breaks exactly in the points of $L$. We are going to define a characteristic vector of $L,$ which for each point of $P$ encodes whether  the rubber band passes above or below the nail.

Let us give a formal definition. Fix a point $P_i \in M$. Let $P_{a_{l}}$ and $P_{a_r}$ be the two points of the polyline $L$ such that $a_l<i<a_r$ and $r-l$ is minimal possible. Depending on whether or not $P$ is itself a vertex of the polyline,  $r-l$ is either 2 or 1.

Let $P$ be a vertex from $M\setminus L$. We say that the polyline  \emph{passes above} $P$ if  $AP$ intersects $P_{a_{l}}P_{a_r}$, and \emph{passes below} $P$ if $AP$ does not intersect $P_{a_{l}}P_{a_r}$.
Let $P$ be a vertex from $ L$. Then the definition is opposite (and it is useful to keep in mind the rubber band analogy here). We say that the polyline  \emph{passes above} $P$ if the segment $AP$ does not intersect with the segment $P_{a_{l}}P_{a_r}$, and \emph{passes below} $P$ if $AP$ intersects with $P_{a_{l}}P_{a_r}$.
Finally, we define the  \textit{characteristic vector} $\chi_L \in \{0,1\}^n$ of $L$, where $\chi_i = 1$ iff $L$ passes above $P_i$.


In what follows, we tacitly assume that a polyline $L$ in a given angle $XYZ$ starts with $X$ and ends with $Z$. The {\it internal vertices} of a polyline are its vertices excluding $X,Z.$
\begin{lemma}\label{main-lemma}
Fix an angle $BAC$ and a set of points $M=\{P_1, P_2, \ldots, P_n\}$ inside the angle. Then any vector in $\{0,1\}^n$ is a characteristic vector of exactly one polyline $L$ in $M$.
\end{lemma}

\begin{proof}
Observe that there are exactly $2^n$ different polylines and $2^n$ different characteristic vectors. Thus, it is sufficient to prove that each vector from $\{0,1\}^n$ is a characteristic vector for some polyline.

Using the rubber band intuition, the proof is obvious: pass the band above/below the $i$-th nail depending on the value of $\chi_i$. Then we are naturally obtaining some polyline in $M$. Let us give a formal proof.

We prove the statement by induction on $n$. For $n=0$ the statement is evident: the polyline $BC$ is the single possible polyline.

Assume that the statement holds for $m<n$, and let us prove it for $n$. Let $M$ be a set of $n$ points inside the angle $BAC$.
Take $\chi \in \{0,1\}^n$ and let us construct the corresponding polyline. Let $P_k\in M$ be the point such that the angle $ABP_k$ is the smallest among $ABP_i$, $i=1,\ldots, n$. We consider two cases depending on $\chi_k$. 

Let $\chi_k = 0$. Then the polyline should pass below $P_k$. Consider a vector $\chi'$ obtained by erasing the $k$'th coordinate of $\chi$. By induction, applied to the angle $BAC$, the set $M'= M \setminus P_k$, and the vector $\chi'$, there exists a polyline $L$ such that $\chi'$ is its characteristic vector (relative to $M'$). Moreover, each polyline for $M'$ clearly passes below $P_k$ and thus, when thought of as a polyline for $M$,  the $k$-th component of the characteristic vector of $L$ equals $0$. This completes the proof in this case. 

Let $\chi_k = 1$. The required polyline should pass above $P_k$. Let $\chi^1$, $\chi^2$ be the vectors consisting of the first $k-1$ and the last $n-k$ coordinates of $\chi$, respectively. By induction, for the angle $BAP_k$ and the set of points $M^1 = \{P_1, P_2,\ldots, P_{k-1}\}$ there exists a polyline with a characteristic vector equal to $\chi^1$. Similarly, for the angle $P_kAC$ and the set  $M^2 = \{P_{k+1}, P_{k+2},\ldots, P_{n}\}$ there exists a polyline with characteristic vector $\chi^2$. By the choice of $P_k$, the union of these polylines passes above $P_k$ and, consequently, its characteristic vector with respect to $BAC$ equals $\chi$.
\end{proof}

\begin{proof}[Proof of the conjecture]

Consider a set of points $M$. We prove the theorem by induction on the number of points of $M$. The case when the points are in convex position serves as the base of induction.

We call the points of $M$ inside  $\co(M)$ red, and the vertices of the convex hull black (here and in what follows, the colors are for convenience of notation only). Fix a black point $A$. For each red point $P$ choose a `green' point $P'$ on the ray $AP$ so that the black and the green points are in convex position. In order to prove the theorem, it is sufficient to show that the number of triangulations of a convex polygon formed by the black and green points (we call these triangulations green) is not greater than the number of partial triangulations of $M$ (we call these triangulations red).


Obviously, any red triangulation contains the edges of $\co(M)$, and similarly for the green triangulations. We are going to compare the number of red and green triangulations subclass by subclass, defined by the exact set of edges drawn from $A$.

Let $B$ and $C$ be the vertices on the convex hull of $M$ that are adjacent to $A$. Fix an arbitrary subset $S$ of green and black points. 
Consider the green triangulations in which $A$ is adjacent precisely to the elements of $S$. In those triangulations, neighboring points  of $S$ (in the left to right order according to the angle $BAC$) are obviously adjacent. These edges together form a polyline $L_1$ from $B$ to $C$ that has $S$ as the set of internal vertices. Also note that $L_1$ passes below the points from $S$ and passes above all other green or black points. To finish the green triangulation, we have to triangulate several non-overlapping polygons below $L_1$. The number of such triangulations is a product of several numbers $W_x$.

By Lemma~\ref{main-lemma}, there exists a polyline $L_2$ from $B$ to $C$ with red and black vertices inside the angle $BAC$ that has the same characteristic vector as $L_1$. (Actually, $L_1$ and $L_2$ share the same set of black vertices.) Let $T$ be the set of internal vertices of $L_2$. In the red triangulations, we draw an edge between $A$ and a red vertex if and only if it belongs to $T$ (cf. Fig.~\ref{fig}). This definition, combined with Lemma~\ref{main-lemma}, guarantees that the set of triangulations that we construct shall be distinct for distinct $S.$ Indeed, we can graphically represent it as follows: $N_{green}(A)=S\leftrightarrow L_1\leftrightarrow \chi(L_1) = \chi(L_2)\leftrightarrow L_2\leftrightarrow T=N_{red}(A)$, where $N_{color}(A)$ is the set of neighbors of $A$ within the subpolygon in the triangulation of the corresponding color, and the element on one side of any arrow  uniquely determines the element on the other side. Note that if we think of the polylines $L_1$ and $L_2$ as of rubber bands, we have that a red point and the corresponding green point are either both above their rubber bands or both below them. 

The last step of the proof is to show that there are at least as many red triangulations with $N_{red}(A)=T$ as green triangulations with $N_{green}(A)=S$. 
Let us cut $L_2$ into the maximal upwards convex (cap) sections that share endpoints.  Consider one such section $D_0,\ldots,D_m$ (see Figure~\ref{fig} for an illustration). Note that the points $D_0$ and $D_m$ are  elements of $S\cup \{B,C\}$. Indeed, if $D_0$ ($D_m$) is different from $B,C,$ then $L_2$ passes below $D_0$ by the fact that the section is a maximal cap. At the same time, $L_1$ passes below a green or a black vertex iff it is in $S$. Define a set $M'$ that consists of $D_0,D_m$ and the set of all red or black points s.t. $L_2$ passes above them (note that $L_2$ passes above any black point inside the angle $D_0AD_m$). The set $\co (M')$ is bounded by $L_2$ on one side, and so we can use any triangulation of $M'$ in the red triangulation.
Assume that there are $c$ black and $a$ green points (strictly) inside the angle $D_0AD_m$, and that $b$ green points among them belong to $S$. (Note that there are no black points that lie inside this angle and that belong to $S$ because of the convexity of $D_0,\ldots, D_m$.)   Within the angle $D_0AD_m$, the partially constructed green triangulation consists of $b+1$ non-overlapping convex polygons that we need to further triangulate. Let $k_1,\ldots,k_{b+1}$ be the sizes of these polygons. Each of the $b$ points from $S$ belongs to two of these polygons, and so  $k_1+\ldots+k_{b+1} = a+b+c+2$. Using Lemma~\ref{lemw}, the number of ways to triangulate those polygons is
\begin{equation}
\prod_{i=1}^{b+1} W_{k_i}\leq W_{k_1+\ldots+k_{b+1} - 2b} = W_{a-b+c+2}\,.
\end{equation}
On the other hand, inside the angle $D_0AD_m$ there are exactly $a$ red points, of which $b$ points correspond to the green points from $S$. They are above $L_2$ and thus do not belong to $M'$. This gives $|M'| = a-b+c+2$. Thus, by the induction statement, the number of partial red triangulations of $M'$ is at least $W_{a-b+c+2},$ which is at least the number of  green triangulations inside the angle $D_0AD_m$.

\begin{figure}
    \centering
\includegraphics[width=0.6\linewidth]{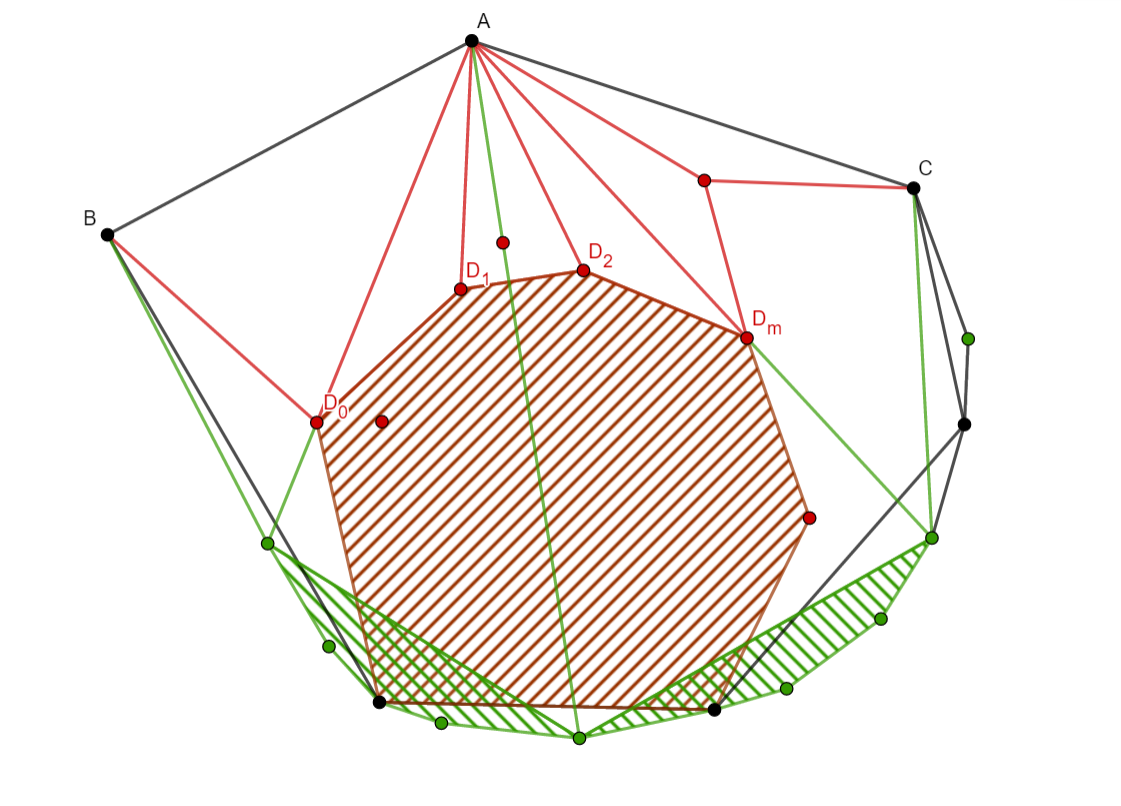}
    \caption{An example of the correspondence between green triangulations and red triangulations for one maximal cap.}
    \label{fig}
\end{figure}

Apply this argument  on each convex section of $L_2$. Once we fix a particular triangulation of each of the aforementioned green polygons, the green triangulation is complete. At the same time, even after fixing the triangulations of all $M'$, the red triangulation may be incomplete: for instance, there may be a polygon below the convex hulls of $M'$s that was not yet triangulated. But there is at least one way to finish the triangulation, and the obtained triangulations will be distinct for distinct incomplete triangulations. 

Thus, the number of green triangulations with $N_{green}(A)=S$ is at most the number of red triangulations with $N_{red}(A)=T$. Moreover, for different sets $S$, the families of corresponding triangulations, both red and green, are disjoint. Summing up the numbers of the green and the corresponding red triangulations, we obtain that the number of red triangulations is at least  the number of green triangulations, which concludes the proof.
\end{proof}

\section{Extremal point sets: proof of Theorem~\ref{thm1}}\label{sec2}
We first prove that the quasi-convex sets are indeed extremal.

\begin{remark}
Let $M$ be an quasi-convex set. Recall that each side $A_iB_i$ of the convex hull has at most $1$ point  $P_i$ that is close to it. We therefore can define a polygon $\mathcal{M}$ with the set of vertices $M$ and a natural order of the vertices: the vertices of $\co (M)$ appear in the counter-clockwise order, and the vertex $P_i$ that is close to a side $A_iB_i$ appears between the vertices $A_i$ and $B_i$.
In what follows, we call $\mathcal{M}$ a \textit{quasi-convex polygon}.
\end{remark}

\begin{lemma} \label{lemma-verify-equality-case}
Let $M$ be a set of $n$ points in a general position. If $M$ is quasi-convex then the number of partial triangulations of $M$ is equal to $W_n$.
\end{lemma}

\begin{proof}
We prove the statement by induction on $n$. The base case is $n=3$. Recall that by convention $W_2 =1$ as well. First note that the following formula is a restatement of the formula for the Catalan numbers in terms of $W$'s:
$$W_n = \sum_{k=2}^{n-1} W_kW_{n-k+1}.$$
Next, let us show that the number of triangulations of a set $M$ of $n$ points in convex position is equal to $W_n$ for all $n\ge 4.$ Indeed, pick any side $AB$ of the convex hull. In any triangulation, there should be a unique vertex $P$ that forms a triangle with $AB.$ Triangle $PAB$ splits $M$ into two convex sets $M_1, M_2$, each of size at least $2$. If $|M_1| = k$ then $|M_2|=n-k+1.$ By induction, the number of triangulations of $M_1$ is $W_k$, and that of $M_2$ is $W_{n-k+1}$. Of course, any triangulation of $M_1$ can be composed with any triangulation of $M_2$ in order to get a triangulation of $M$. Varying the choice of $P,$ we get that the number of triangulations of $M$ is $\sum_{k=2}^{n-1} W_kW_{n-k+1}$, which is equal to $W_n$.

In what follows, we assume that $M$ has at least $1$ interior point.

Let $A$, $B$ be some non-neighbor vertices of $\mathcal{M}$. Let us prove that the segment $AB$ is fully contained inside the polygon $\mathcal{M}$. It is sufficient to show that $AB$ does not intersect any side of $\mathcal M$. This is obvious in the case when both vertices of the side belong to the convex hull of $M$. Now assume that $CD$ is a side of $\mathcal M$ such that $C$ is an interior point of $M$ and $D$ is a vertex of its convex hull. 
Denote $E$ the other vertex of the side of $\co (M)$ to which $C$ is close. 
Then, by the definition of a close point, $C \in \co(\{B,D,E\})\cap \co(\{A,D,E\})$. But $AB$ cannot be contained in both triangles. Indeed, w.l.o.g. assume that the distance from $A$ to the line $DE$ is at most the distance from $B$ to the line $DE.$ Then the distance from any point of $AB$ to $DE$ is at least as big as the distance from $A$ to the line $DE$, and thus $AB$ is not contained in $ADE.$ At the same time, the segment $CD$ lies inside $ADE,$ which means that the triangle $ADE$ ``separates'' $CD$ and $AB$. Therefore, $CD$ does not intersect with $AB$. Since this is valid for any choice of $CD,$ we conclude that $AB$ lies fully inside  $\mathcal{M}$.

Fix an interior point $P$ of $M$. Let $A$, $B$ be the two neighbor points of $P$ in $\mathcal{M}$. Let us consider a partial triangulation of $\mathcal{M}$. If $P$ is an isolated vertex of the triangulation then this triangulation is also a partial triangulation of $M \setminus \{P\}$, and by induction there are $M_{n-1}$ of those. Note that $W_{n-1} = W_{n-1}W_2$.

If $P$ is not an isolated vertex then, by the closeness of $P$, $PA$ and $PB$ must be among the edges of the triangulation. Let $C$ be a vertex of a triangle in the triangulation that contains $AP$. The triangle $ACP$ divides $\mathcal{M}$ into two polygons $\mathcal{M}_1$ and $\mathcal{M}_2$. Each of these polygons are not self-intersecting by the conclusion derived two paragraphs above. Denote by $M_1$, $M_2$ the corresponding point sets. Let us prove that $M_1$ and $M_2$ are quasi-convex. First we note that interior points different from $P,C$ remain close to the corresponding sides. The fact that $P$ is contained in the triangle $ABC$ implies that $P$ is a vertex of the convex hull of $M_2.$ Thus, the only potential problematic vertex is $C$ itself. If $C$ is a vertex of the convex hull of $M$ then both  $ M_1$ and $ M_2$ are quasi-convex. If $C$ is a close point to some side $DE$ then $C$ is inside the triangle $PDE$. Therefore, $C$ is a vertex of the convex hull of $M_1$ and $M_2$, which means that $M_1$ and $M_2$ are quasi-convex again.

If $M_1$ consists of $k$ points then $M_2$ consists of $n-k+1$ points. Moreover, note that, by varying the choice of $C$, $k$ can take each value from $2$ to $n-2$ ($M_2$ has at least $3$ vertices: $B,C,P$).

By the inductive assumption, the number of partial triangulation of $M$, where $P$ is not an isolated vertex and $\mathcal{M}_1$ consists of $k$ vertices, is $W_{k} \cdot W_{n-k+1}$. Adding the number of partial triangulations with $P$ isolated, we get that the total  number of partial triangulations of $M$ is equal to
$$W_{n-1}+\sum_{k=2}^{n-2}W_kW_{n-k+1} = \sum_{k=2}^{n-1} W_kW_{n-k+1}  = W_n,$$ which concludes the proof.
\end{proof}

In what follows, we show that if $M$ contains interior points that are not close to any of the sides of the convex hull, then the number of partial triangulations of $M$ is strictly bigger than $W_n.$ We start with a set of definitions and a key lemma that are analogous to that in Section~\ref{sec1}.

Consider some points $A,X,Y$ in general position on the plane, so that neither $AX$ nor $AY$ is vertical.

We say that a point $X\in M$ is \emph{to the left} of a point $Y\in M$ if, when rotating a vertical ray emanating from $A$ in the counter-clockwise direction,  we first hit $X$ and then $Y$ (cf. Figure~\ref{fig3}).

\begin{figure}
    \centering
\includegraphics[width=0.6\linewidth]{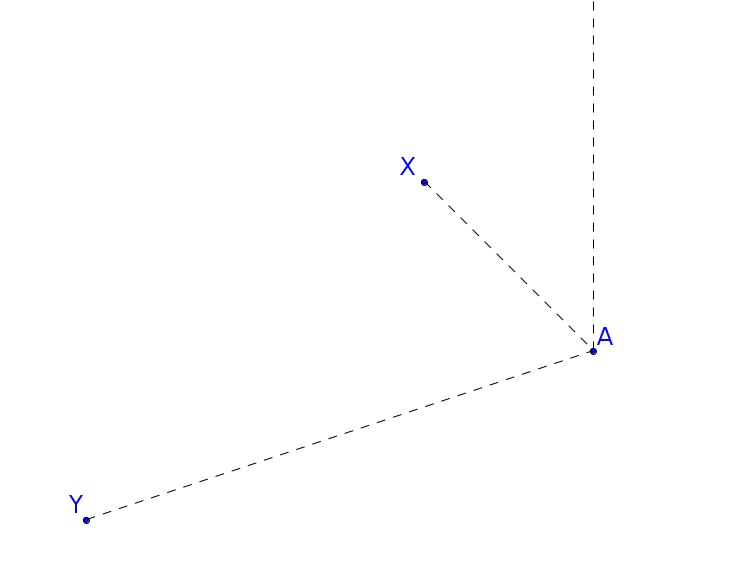}
    \caption{Point $X$ is to the left of the point $Y$.}
    \label{fig3}
\end{figure}

In what follows, we use the term {\it polygon} for a closed polygonal chain. It will play the same role as polylines in the first proof.  
We only consider polygons with vertices $P_1,\ldots, P_n$ such that for any two consecutive vertices $P_i,P_{i+1}$ (indices taken modulo $n$) we need to rotate the ray $AP_i$ by less than $\pi$ counter-clockwise in order to hit $P_{i+1}$. We call such polygons {\it good}, but often omit it, writing simply a ``polygon'').

Fix a point  $A$ and a set of points $M$, all points being in general position. 
The set $M = \{P_1, P_2, \ldots, P_n\}$ is enumerated left to right. We also assume that no ray $AP_i$ is vertical. 
Consider an arbitrary polygon $L=P_{a_1}\ldots P_{a_{m}}$, enumerated from left to right.

The next definition is  describing the position of points in $M$ with respect to $L$. Let us first give it informally. Imagine that there are nails in each of the points $M$ and there is a closed rubber band stretched around $A$ that breaks exactly in the points of $L$. We are going to define a characteristic vector of $L,$ which for each point of $P$ encodes whether  the rubber band passes closer to $A$ than the nail or farther from $A$ than the nail.

Let us give a formal definition. Fix a point $P_i \in M$. Let $P_{l(i)}$ and $P_{r(i)}$ be the two points of the polygon $L$ (i.e., $l(i),r(i)\in \{a_1,\ldots, a_{m}\}$) such that $l(i)<i<r(i)$ and $r(i)-l(i)$ is minimal possible, all indices taken modulo $n$. 

Let $P_i$ be a vertex from $M\setminus L$. We say that the polyline  \emph{passes above} $P_i$ if  $\angle P_{l(i)}P_iP_{r(i)}<\pi$, and \emph{passes below} $P_i$ if 
$\angle P_{l(i)}P_iP_{r(i)}>\pi$, where here and below we measure the angles that contain $A$.
Let $P_i$ be a vertex from $ L$. Then the definition is opposite (and it is useful to keep in mind the rubber band analogy here). We say that the polygon  \emph{passes above} $P_i$ if $\angle P_{l(i)}P_iP_{r(i)}>\pi$ 
and \emph{passes below} $P_i$ if $\angle P_{l(i)}P_iP_{r(i)}<\pi$.
Finally, we define the  \textit{characteristic vector} $\chi_L \in \{0,1\}^n$ of $L$, where $\chi_i = 1$ iff $L$ passes above $P_i$.


Consider the map $\Psi: \{\text{polygons on }M \}\to \{0,1\}^n$ that maps a polygon to its characteristic vector.
\begin{lemma}\label{main-lemma2}
Fix a point $A$ and a set of points $M=\{P_1, P_2, \ldots, P_n\}$ inside the angle with properties as above. Then $\Psi$ is an injection. Moreover, the image of $\Psi$ stays the same if we replace the set $M$ by another set of points $M' = \{P_1',\ldots, P_n'\}$, where $P_i'$ lies on the ray $AP_i$ for each $i =1,\ldots, n.$
\end{lemma}

\begin{proof}
Using the rubber band intuition, the proof is simple. For the first part, if we pass  the band above/below the $i$-th nail depending on the value of $\chi_i$ and then tighten it, then we clearly get a unique polygon (it may not necessarily be good though). Once we got a good polygon, we can move the nails along the rays emanating from $A$. It may change the vertices of the polygon, but it will not change the characteristic vector of the polygon, and the polygon will stay good. Let us give a formal proof.

We first prove the second part of the statement. In order to do so, it is sufficient to show the following: whenever $\Psi(M)$ contains a certain characteristic vector, this vector is also contained in $\Psi(M'),$ where $M' = \{P_1,\ldots, P_{i-1},P_i',P_{i+1},\ldots, P_n\}$ and $P'_i$ lies on the ray $AP_i$.

We start with the set $M$ and fix a polygon $L$ that corresponds to some characteristic vector $\chi$. We move the point $P_i$ towards $P_i'$. Note that the characteristic vector of the polygon changes only if for some $P_j$ the points $P_j$, $P_{l(j)}$ and $P_{r(j)}$ become collinear. 
(This is easy to see because the definition of above/below the polygon for a given point $P_j$ depends on the positioning of $P_{\ell(j)}$ and $P_{r(j)}$ only, and thus, in order to see the change in the characteristic vector of $L$, we need the collinearity to happen for such three points for some $j$.) Moreover, one of these three points must be $P_i.$
Such collinearities will occur only finitely many times while moving $P_i$ to $P_i'$, and so it is enough to check that the vector $\chi$ is a characteristic vector of some polyline for the new point set, obtained after $P_i$ passing through this collinearity. Thus, abusing notation, let us assume that $P_i$ is coincides with the position of the point just before the collinearity and $P_i'$ coincides with the position of the point just after the collinearity.

We also denote $P_k':= P_k$ if $k\ne i$ for convenience. We consider some cases depending on the position of the points. Let $\chi_j$ be the $j'$th coordinate of $\chi.$  If $\chi_j=1$ and $P_j$ is a vertex of $L$ then  $\angle P_{l(j)}P_j P_{r(j)}<\pi$ and  $\angle P_{l(j)}'P_j' P_{r(j)}'>\pi$. In this case, we remove the vertex $P_j$ from the vertices of $L$, obtaining a new polygon $L'$. It is easy to see that the characteristic vector of $L'$ (as a polygon on $M'$) coincides with $\chi$. Indeed, the $j$-th coordinate coincides because of the modification we made, and the other coordinates stay the same because we chose $P_i, P_i'$ sufficiently close to the collinearity position (and thus no other collinearity happened while moving from $P_i$ to $P_i'$).
If $\chi_j = 1$ and $P_j$ is not a vertex of $L$ then $\angle P_{l(j)}P_j P_{r(j)}>\pi$ and  $\angle P_{l(j)}'P_j' P_{r(j)}'<\pi$. In this case, we add the vertex $P_j$ to the vertices of $L$, obtaining a new polygon $L'$. The analysis is analogous. The cases when $\chi_j=0$ are treated analogously. This completes the proof of the second part.

We move on to the proof of the first part of the lemma. We first note that, when replacing $M$ by $M'$ as in the second part of the lemma, any ``good'' polygon stays good (this only depends on the angles $P_{a_i}AP_{a_{i+1}}$, but these angles stay the same when moving the points along the rays emanating from $A$). Thus, when we move points along the rays, neither the set of good polygons nor the image of $\Psi$ change. By moving the points along the rays, we can achieve that $M$ is in convex position, and thus it is sufficient to prove the first part of the lemma for such sets $M$.

Assume that $M$ is in convex position and take two different polygons $L_1,L_2$ on the vertices of $M$. Going left to right, let $P_{a_i}$ be the first point that is the vertex of exactly one polygon of those two, say, of $L_1$. Using convexity of $M$, the point $P_{a_i}$ lies above the polygon $L_1$ and below the polygon $L_2$, and thus these two polygons have different characteristic vectors. This proves the injectivity of $\Psi.$
\end{proof}

The first and key step in the proof of Theorem~\ref{thm1} is to reduce the situation to the case when we have one internal point. Assume that $A\in M$ is an internal point that is not close to any side of $\co (M)$. Recall that we call the points of the convex hull of $M$ {\it black}. We are going to apply the same argument to $A$ as it was done in the proof of the conjecture. More precisely, we replace each interior (red) point $P\in  M\setminus \{A\}$ by another (green) point $P'$, so that all green and black points are in convex position. Denote by $M'$ the set of all green and black points together with $A$. The only difference between this and the previous argument is that we apply Lemma~\ref{main-lemma2} instead of Lemma~\ref{main-lemma}. Lemma~\ref{main-lemma2} guarantees the bijection between good green polygons and good red polygons, and this is the only property that we actually use. The rest of the proof stays precisely the same, and we conclude that the number of partial triangulations of $M$ is at least the number of partial triangulations of $M'$.

The next step is the following lemma.
\begin{lemma}
$A$ is not close to any side of the convex hull of $M'.$
\end{lemma}
\begin{proof}
Assume that $A$ is close to a side $PQ$ of $\co (M')$, and $P$ precedes $Q$ in the counter-clockwise order. We consider two cases depending on the color of $P,Q.$ If one of these points is green, then consider the black points $P',Q'$ on the convex hull such that $P'$ is the closest black point that precedes $P$ in the counter-clockwise order (which may coincide with $P$) and, similarly, $Q'$ is the closest black point that succeeds $Q$ in the counter-clockwise order (which, again, may coincide with $Q$). Note that $PQ$ separates $A$ from $P'Q'$ because $A$ was contained in the convex hull of the black points, and thus either $PP'Q'$ or $P'Q'Q$ is a triangle that does not contain $A$.

Next, assume that both $P$ and $Q$ are black. Since $A$ was not close to $PQ$ in $M$, there was a red point $U\in M$ such that $A$ was not contained in the triangle $UPQ$. This implies that one of the lines $UQ$, $UP$ separated $A$ from $UPQ$. W.l.o.g., assume that this was $UQ$. Denote by $U'$ the green point that corresponds to $U$. But then the line $U'Q$ separates $A$ from the triangle $U'PQ,$ and thus $A$ is is not close to $PQ.$
\end{proof}

The last part of the proof is to show that if $M'$ is a set of $n$ points with exactly $1$ interior point $A$ that is not close to any side of the convex hull, then the number of partial triangulations of $M'$ is strictly bigger than $W_n$.

Assume that the vertices of $M'$ are $P_1,\ldots, P_{n-1}$. First, note that if a point $X$ is a close to some side $P_iP_{i+1}$ if and only if it is contained in the triangles $P_iP_{i+1}P_{i+2}$ and $P_{i-1}P_iP_{i+1}$ (with indices modulo $n-1$). Thus, for any $4$ consecutive vertices $P_{i-1},P_i,P_{i+1},P_{i+2}$ $A$ is either not contained in $P_{i-1}P_iP_{i+1}$ or in $P_iP_{i+1}P_{i+2}$. Fix any vertex $P_i$ of the convex hull of $M'$ such that $A$ does not lie in $P_{i-1}P_iP_{i+1},$ and assume that the ray $P_iA$ intersects the side $P_jP_{j+1}$ of the convex hull of $M'$. By symmetry, we may assume that $A$ is not contained in $P_{j-1}P_jP_{j+1}.$ Let $A'$ be a point on the ray $P_iA$ just past the segment $P_{j}P_{j+1}$. In particular, $P_1,\ldots, P_{n-1},A'$ are in convex position.

Recall that in the proof of the conjecture, we constructed an injection from the set of green triangulations to the set of red triangulations. Therefore, in order to show that $M'$ has strictly more than $W_n$ partial triangulations, it is sufficient to exhibit a partial red triangulation that is not an image of any green triangulation.

Consider the red triangulation, in which $P_i$ is only connected to $P_{i-1},P_{i+1}$ and $A$, and in which $A$ is connected to all $P_s$, $s\in [n-1]\setminus \{j\}$. We additionally draw an edge $P_{j-1}P_{j+1}.$ We claim that this red triangulation is not an image of any green triangulation.

Indeed, in the proof of the conjecture we split the neighbors of $P_i$ into convex parts, which in our case would be $P_{i-1}A$ and $AP_{i+1}.$ In the green triangulation, we will then triangulate the polygons $P_{i-1}P_{i-2}\ldots P_jA'$ and $P_{i+1}P_{i+2}\ldots P_{j+1}A$ independently, which corresponds to red triangulations that each contain the edges $P_jA$ and $P_{j+1}A$ (and a triangle $P_jP_{j+1}A$). However, there is no edge $P_jA$ in our triangulation, which proves the claim and concludes the proof of Theorem~\ref{thm1}.\\

{\sc Acknowledgements: } This research was started during the 2nd Adygea Problem Solving Workshop in October 2020. The authors thank Emo Welzl for sharing this problem. The authors acknowledge the financial support from the Russian Government in the framework of MegaGrant no 075-15-2019-1926.


\begin{thebibliography}{10}
\bibitem{problist} {\it A collection of problems from the Oberwolfach workshop on Discrete Geometry}, October 2020, \url{https://page.mi.fu-berlin.de/rote/Kram/Problems-Discrete-Geometry-2020.pdf}
\bibitem{WW} U. Wagner, E. Welzl, {\it Connectivity of Triangulation Flip Graphs in the Plane} (2020), arXiv:2003.13557

\bibitem{W} E. Welzl, {\it The Number of Triangulations on Planar Point Sets}, Kaufmann, M., Wagner, D. (Eds.), Graph Drawing, Lecture Notes in Computer Science. Springer Berlin Heidelberg, Berlin, Heidelberg (2007),  1--4.

\end{thebibliography}
\end{document}